\documentclass[11pt,a4paper]{amsart}       
\usepackage{graphicx}
\usepackage{amsmath,amssymb,amsthm}

\DeclareMathOperator{\dist}{dist}
\newtheorem{nonsec}[equation]{}
\newtheorem{thm}[equation]{Theorem}
\newtheorem{lem}[equation]{Lemma}
\newtheorem{rem}[equation]{Remark}
\newtheorem{prop}[equation]{Proposition}
\numberwithin{equation}{section}

\begin{document}

\title{A new intrinsic metric and quasiregular maps}

\author[M. Fujimura]{Masayo Fujimura}
\address{Department of Mathematics, National Defense Academy of Japan, Japan}
\email{masayo@nda.ac.jp}

\author[M. Mocanu]{Marcelina Mocanu}
\address{Department of Mathematics and Informatics, Vasile Alecsandri
         University of Bacau, Romania}
\email{mmocanu@ub.ro}

\author[M. Vuorinen]{Matti Vuorinen}
\address{Department of Mathematics and Statistics, University of Turku,
         Turku, Finland}
\email{vuorinen@utu.fi}

\keywords{Hyperbolic metric, Triangular ratio metric,
          Quasiconformal map, Quasiregular map}
   \subjclass[2010]{30C20, 30C15, 51M99}
\date{}

\begin{abstract}
  We introduce a new intrinsic metric
in subdomains of a metric space and give upper and lower bounds for it in terms
of well-known metrics.
  We also prove  distortion results for this  metric 
  under quasiregular maps.
\end{abstract}

\maketitle

\section{Introduction}

\label{section1}
\setcounter{equation}{0}

Distance functions specific to a domain
$G \subset \mathbb{R}^n, n\ge2\,,$ or, as we call them,
intrinsic metrics, are some of the key notions of geometric function theory
and are currently studied by many authors.
See for instance the recent monographs \cite{gh,gmp,hkv,p} and
papers \cite{bh,dhv,fmv,h1,h2,himps}.
In \cite{gh} intrinsic metrics are used as a powerful tool to analyse
the properties of quasidisks and \cite{hkv} provides a survey of some
recent progress in the field.
A list of twelve metrics recurrent in geometric function theory
is given in \cite[pp. 42-48]{p}.

In the classical case $n=2$ one can define the hyperbolic metric
of a simply connected domain by use of a conformal mapping given
by the Riemann mapping theorem and the hyperbolic metric of
the unit disk \cite{bm}.
This metric is conformally invariant and therefore a most useful tool.
For  dimensions $n\ge 3,$ by Liouville's theorem, conformal mappings
$f \colon D \to D'$ of domains $D, D' \subset \mathbb{R}^n$ are of the form
$f= g | D$ where $g$ is a M\"obius transformation \cite[pp. 64-75]{gmp}, 
\cite[pp.11-12]{hkv},
and therefore there is no counterpart of the Riemann mapping theorem,
and the planar procedure is not applicable.
This state of affairs led many researchers to look for generalized hyperbolic 
geometries and metrics which
share at least some but not all properties of the hyperbolic metric 
\cite[Ch. 5]{hkv}. For instance,
the quasihyperbolic and distance ratio metrics studied in
\cite{bh,gh,gmp,hkv} do not enjoy the full conformal invariance
property for any dimension $n\ge 2,$ both are invariant under
similarity transformations only.

Here we study a function recently used  as a tool by  O. Dovgoshey, P. Hariri,
  and M. Vuorinen  \cite{dhv}
and show that this function satisfies the triangle inequality and, indeed,
defines an intrinsic metric of a domain.
Moreover, we compare it to the distance ratio metric and find
two-sided bounds for it. Finally, we study the behavior of this
metric under quasiconformal mappings.  

For a proper nonempty open subset $D \subset {\mathbb R}^n\,$
and for all $x,y\in D$, {\sl the  distance ratio metric}
$j_D$ 
is defined as
$$
 j_D(x,y)=\log \left( 1+\frac{|x-y|}{\min \{d_{D}(x),d_{D}(y) \} } \right)\,.
$$
For a proof of the triangle inequality,
see \cite[Lemma 3.3.4]{gh}, \cite[7.44]{avv}.
If there is no danger of confusion, we write
$d_D(x)=d(x)= d(x,\partial D)=\dist (x, \partial D)\,.$

In this paper our goal is to prove that
the expression \eqref{dhvfun2}  studied in \cite{dhv}
for $ (X,\rho)=(D,j_D) $
is, in fact, a metric. We also prove several upper and lower
bounds for this new metric.

\begin{thm} \label{dhvfun}
  Let $(X,\rho )$  be a metric space and for  $x,y\in X, c>0\,,$ let
  \begin{equation} \label{dhvfun2}
    W(x,y):=\log \Big(1+2c\sinh \frac{\rho(x,y)}{2}\Big)\,.
  \end{equation}
  If $c\geq 1$, then $W$ is a  metric on $X$.
  Moreover, if $\rho =j_{{\mathbb{B}^2}}$ where ${\mathbb{B}^2}$ 
  is the unit disk, then
  $W$ is a metric on $\mathbb{B}^2$ if and only if $c\geq 1\,.$
\end{thm}

\begin{thm} \label{thm110}
  Let $G$ be a proper subdomain of $\mathbb{R}^n\,.$
  The following inequality holds for all $x,y \in G$
  $$
    \frac{j_G(x,y)}{2} \le \log\Big(1+ 2 \sinh\frac{j_G(x,y)}{2}\Big) \le
    \min \left\{ j_G(x,y), \frac{j_G(x,y)}{2} + \log \frac{5}{4} \right\} \,.
  $$
\end{thm}

\begin{figure}
  \centerline{\includegraphics[width=0.5\linewidth]{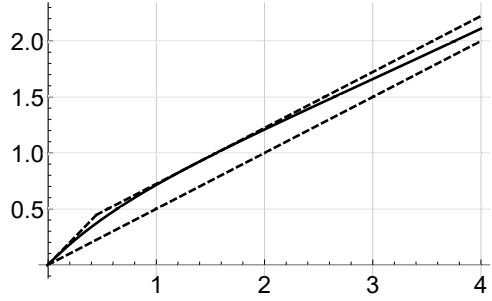}}
  \caption{The graphs of the functions $ y=\frac{t}2, y=\log(1+2\sinh\frac{t}2)$, and
           $ y=\min\{t,\frac{t}2+\log\frac54\}$ in Theorem \ref{thm110}.}
  \label{fig:1}
\end{figure}
We conclude our paper by studying the behavior of the metric 
of Theorem \ref{dhvfun} under 
quasiregular mappings defined on the unit disk and prove the following
result, which is based on a recent sharp version of the Schwarz lemma for 
quasiregular mappings for $n=2$ \cite{wv}.  

\begin{thm} \label{Wqc}
 Let $f\colon \mathbb{B}^{2}\rightarrow \mathbb{B}^{2}$ be a non-constant 
 $K$-quasiregular mapping, where $K\geq 1$.
 Denote by $\rho =\rho _{\mathbb{B}^{2}}$ the hyperbolic metric and
 let $W_{\lambda }\left( x,y\right)
 =\log\left( 1+2\lambda \sinh \frac{\rho \left( x,y\right) }{2}\right) $,
 where $ \lambda \geq 1$ and let $c(K)$ be the constant in 
 Theorem {\bf \ref{NewThm}}.
 For all $x,y\in \mathbb{B}^{2}$
 \begin{equation}
   W_{\lambda }\left( f\left( x\right) ,f\left( y\right) \right)
   \leq 2\lambda c(K)
   \max \left\{ W_{\lambda }\left( x,y\right) ^{1/K},
               W_{\lambda }\left(x,y\right) \right\} \label{distow}.
 \end{equation}
\end{thm}



\section{Preliminaries}

We recall the definition of  the hyperbolic
distance $\rho_{\mathbb{B}^n}(x,y)$ between two points
$x,y \in \mathbb{B}^n = \{x \in \mathbb{R}^n: |x|<1 \}$
\cite[Thm 7.2.1, p. 130]{be}:
\begin{equation}  \label{eq:tro}
  \tanh {\frac{\rho_{\mathbb{B}^n}(x,y)}{2}}
     =\frac{|x-y|}{\sqrt{|x-y|^2+(1-|x|^2)(1-|y|^2)}}\,.
\end{equation}
One of the main properties of the hyperbolic metric is its invariance
under a M\"obius self-mapping
$T_a\colon \mathbb{B}^n \to \mathbb{B}^n\,, $
with $T_a(a) = 0\,, |a|<1\,,$ of the unit ball $\mathbb{B}^n\,. $
 In other words, the mapping $T_a$
is an isometry.
By \cite[p.35]{be} we have for
$ x,y\in \mathbb{H}^n = \{ z \in \mathbb{R}^n: z_n>0\}$
\begin{equation}\label{cro}
  \cosh{\rho_{\mathbb{H}^n}(x,y)}=1+\frac{|x-y|^2}{2x_ny_n}\,.
\end{equation}

For $ D \in \{ \mathbb{B}^n, \mathbb{H}^n \}$ and all $x, y \in D$ we have by \cite[Lemma 4.9]{hkv}
\begin{equation}  \label{eq:jrho}
  j_D(x,y) \le \rho_D(x,y) \le 2 j_D(x,y) \,.
\end{equation}

By means of the Riemann mapping theorem one can extend
the definition of the hyperbolic metric to the case of simply
connected plane domains \cite[Thm 6.3, p. 26]{bm}.

\section{A new metric} \label{section3b}

\begin{thm}\label{th4} {\rm \cite{dhv}}
  Let $D$ be a nonempty open set in a metric space $(X, \rho)$ and
  let $\partial D \neq\varnothing$. Then the function
  $$
    h_{D,c}(x,y) = \log\left(1+c\frac{\rho(x,y)}{\sqrt{d_D(x)d_D(y)}}\right)\,,
  $$
  is a metric for every $c \ge 2$. The constant $2$ is best possible here.
\end{thm}

This metric is listed in \cite{dd} and it has found some applications
in \cite{nl}.


\begin{prop}  {\rm \cite[Prop. 2.7]{dhv}} \label{dhvProp}
  For $c, t>0$, let
  $$ F_c(t)=\log\left(1+2c \sinh{\frac{t}{2}} \right).$$
  Then the double inequality
  $$
     \frac{c}{2(1+c)}t< F_c(t) < ct
  $$
  holds for
  $c\geq \frac{1}{2}$ and $t>0\, . $
\end{prop}

\begin{lem}\label{dhv44} {\rm\cite[Lemma 4.4]{dhv}}
  Let $D$ be a proper subdomain of $\mathbb{R}^n$.
  Then for $c>0$ and   $x,y \in D$
  $$
    \log\Big(1+2 c \sinh\frac{j_D(x,y)} {2}\Big)
    \le h_{D,c}(x,y)  \le c j_D(x,y)\,.
  $$
\end{lem}

We will now prove that the expression on the left hand side of
the inequality of Lemma \ref{dhv44} satisfies the triangle inequality
and for that purpose we need the following refined form of Proposition
\ref{dhvProp} for $c \ge 1\,.$
This refined result and some of the lower bounds that will be proved
below for the function $F_c$ in Proposition \ref{dhvProp},
also lead to improved constants in some of the results of \cite{dhv}.

\begin{lem} \label{propmm}
  The function ${F_c(t)}/{t}$
  is decreasing from $(0,\infty) $
  onto $(1/2,c)\,$ if and only if $c\ge 1\,.$
\end{lem}

\begin{proof}
Let $w(t) = 1+2c\sinh (t/2)\,.$
Differentiation yields
$$
  \bigg(\frac{F_c(t)}{t}\bigg)^{\prime }=\frac{1}{t^{2}} g(t)\,,
  \quad g(t):=
  \bigg( \frac{ct\cosh (t/2)}{w(t)}-\log (w(t))\bigg) \,,
$$

$$
  g^{\prime }( t ) =\frac{ct}{2\left( w(t)\right) ^{2}}
  \bigg( \sinh \Big(\frac{t}{2}\Big) -2c\bigg) \,.
$$
The equation $\sinh \left( \frac{t}{2}\right) =2c$
has the unique solution 
$$
  t_{1}=2\log \left( 2c+\sqrt{4c^{2}+1}\right) >0\,.
$$

We have $g^{\prime }(t)<0$ for $0<t<t_{1}$ and $g^{\prime }(t)>0$
for $ t>t_{1}$.
Then $g$ is strictly decreasing on $\left[ 0,t_{1}\right] $ and
strictly increasing on $[t_{1},\infty )$.
Note that $g\left( t\right) <g\left( 0\right) =0$ for $0<t\leq t_{1}$.

Assume that the limit
$L(c):=\underset{t\rightarrow \infty }{\lim }g\left(t\right) \,$
is finite. Then:
\begin{enumerate}
\item[a)]
  if $L\left( c\right) \leq 0$, then $g\left( t\right) <L\left( c\right)
  \leq 0$ for $t>t_{1}$. In this case, $g\left( t\right) <0$ for all $t>0$.
  It follows that ${F_c(t)}/{t}$ is strictly decreasing.
\item[b)] if $L\left( c\right) >0$, then there exist a unique point
  $t_{2}>t_{1}$ such that $g\left( t_{2}\right) =0$.
  In this case $g\left( t\right) >0$ for all $t>t_{2}$.
  It follows that ${F_c(t)}/{t}$ is strictly decreasing on
  $ \left[ 0,t_{2}\right] $ and strictly increasing on $[t_{2},\infty )$.
\end{enumerate}

Now we compute $L(c)\,.$
Setting  $s=t/2$ we see that
$$
  L(c)=\lim_{s\rightarrow\infty}\left( \frac{2cs\cosh (s)}{w(s)}
      -\log(w(s))\right)\,.
$$
Now we use the change of variable
$\frac{1}{\sinh \left( s\right) }=u$ when $s>0\,.$
Then $1+2c\sinh (s)=\frac{u+2c}{u}$ and
$\cosh (s)=\sqrt{1+\frac{1}{u^{2}}}=\frac{\sqrt{u^{2}+1}}{u}$.
Moreover, $\sinh \left( s\right) =\frac{1}{u}>0$
implies $s=\log \left( \frac{1}{u}+\sqrt{1+\frac{1}{u^{2}}}\right) \,. $

Write $ v=2c\dfrac{\sqrt{u^2+1}}{u+2c} $. It follows that
\begin{align}
  L(c)
     &=\lim_{u\rightarrow 0}\left( v\log
          \bigg( \frac{1}{u}+\sqrt{1+\frac{1}{u^{2}}}\,\bigg)
          -\log \Big(1+\frac{2c}{u}\Big)\right)
       \label{1} \\
     &=\lim_{u\rightarrow 0}\left[ v\log
          \left( 1+\sqrt{u^{2}+1}\right) -\log \left( u+2c\right)
       +\left( 1-v\right) \log u\right] \text{,}  \notag
\end{align}
where
$$
    \lim_{u\rightarrow 0}v\log \left( 1+\sqrt{u^{2}+1}\right)
      =\log 2\,, \quad \lim_{u\rightarrow 0}\log \left( u+2c\right)
      =\log \left( 2c\right) \,.
$$

The limit
$ \lim_{u\rightarrow 0}\left( 1-v\right) \log u$ has the indeterminate
form \thinspace $0\cdot \infty $.
But
\begin{align}
  \underset{u\rightarrow 0}{\lim }\left( 1-v\right) \log u
    &=\underset{u\rightarrow 0}{\lim }\left( \frac{\left(
       u+2c\right) ^{2}-4c^{2}\left( u^{2}+1\right) }{\left( u+2c\right) \left(
       u+2c+2c\sqrt{u^{2}+1}\right) }\right) \log u  \label{2} \\
    &=\frac{1}{8c^{2}}\lim_{u\rightarrow 0}\Big(\left(
       1-4c^{2}\right) u^{2}\log u+4cu\log u\Big) \text{.}  \notag
\end{align}
Since $\underset{u\rightarrow 0}{\lim }u\log u=$
$\underset{u\rightarrow 0}{\lim }u^{2}\log u=0$,
(\ref{1}) and (\ref{2}) imply $L(c) =\log 2-\log \left( 2c\right) $.

\medskip

We have $L\left( c\right) \leq 0$ if $c\geq 1$ and $L\left( c\right) >0$
if $0<c<1$.

In conclusion, ${F_c(t)}/{t}$ is strictly decreasing on
$\left( 0,\infty\right) $ if and only if $c\geq 1$ and its limit values at
$0$ and $\infty$ follow easily.
\end{proof}

\bigskip
\begin{nonsec}{\bf Proof of Theorem \ref{dhvfun}.} {\rm
The proof for $c \ge 1$ follows readily from  Lemma \ref{propmm}
and a general property of metrics  \cite[7.42(1)]{avv}. The well-known fact
that $j_G(x,y)$ is a metric is recorded e.g. in \cite[7.44]{avv}.

We next show that for $c \in (0,1)$ the function
$$
  W(x,y) = \log\Big(1 + 2 c \, \sinh \frac{j_{{\mathbb{B}^2}}(x,y)}{2}\Big)
$$
fails to satisfy the triangle inequality in the unit disk ${{\mathbb{B}^2}}\,.$
Write
$$
  E(x,y) = 1 + \frac{|x-y|}{\min \{ 1-|x|, 1-|y|\}} \,.
$$
The inequality
$W\left( x,z\right) >W\left(x,y\right) +W\left( y,z\right) $
is equivalent to
\begin{equation}
  \frac{E\left( x,z\right) -1}{\sqrt{E\left( x,z\right) }}
  >\frac{E\left(x,y\right) -1}{\sqrt{E\left( x,y\right) }}
   +\frac{E\left( y,z\right) -1}{\sqrt{E\left( y,z\right) }}
   +c\frac{E\left( x,y\right) -1}{\sqrt{E\left(x,y\right) }}
    \frac{E\left( y,z\right) -1}{\sqrt{E\left( y,z\right) }}\text{.}
  \label{inJ1}
\end{equation}
Assume that $0<y<z<1$ and $x=-y$. In this case, (\ref{inJ1}) writes as
\begin{align*}
  &\frac{z+y}{\sqrt{\left( 1+y\right) \left( 1-z\right) }} \\
  & \  >\frac{2y}{\sqrt{\left( 1-y\right) \left( 1+y\right) }}
  +\frac{z-y}{\sqrt{\left( 1-y\right)\left( 1-z\right) }}
  +c\frac{2y\left( z-y\right) }{\left( 1-y\right)
    \sqrt{\left( 1+y\right) \left( 1-z\right) }}\text{,}
\end{align*}
which is equivalent to

\begin{equation*}
  H\left( y,z\right) :=p\left( y\right) z+q\left( y\right) \sqrt{1-z}-r\left(
  y\right) <0\text{,}
\end{equation*}
where
$$
  p\left( y\right)
     :=\frac{1}{\sqrt{1-y}}+\frac{2cy}{\left( 1-y\right)\sqrt{1+y}}
        -\frac{1}{\sqrt{1+y}}\,,
  \quad q\left( y\right)
     :=\frac{2y}{\sqrt{\left( 1-y\right) \left( 1+y\right) }}
$$
$$
  r\left( y\right)
    :=y\left(\frac{1}{\sqrt{1-y}}
       +\frac{2cy}{\left( 1-y\right) \sqrt{1+y}}
       +\frac{1}{\sqrt{1+y}}\right) \,.
$$

Note that
$$
  \underset{z\nearrow 1}{\lim }H\left( y,z\right) =p\left( y\right)
  -r\left( y\right) =\sqrt{1-y}+\frac{2cy}{\sqrt{1+y}}-\sqrt{1+y}
$$
and that
$\underset{y\nearrow 1}{\lim }\left( p\left( y\right)
 -r\left( y\right)\right)=\sqrt{2}\left( c-1\right) <0$.
Take $0<a<1$ such that $p\left(a\right) -r\left( a\right) <0$.
Since
$\underset{z\nearrow 1}{\lim }H\left(a,z\right)
 =p\left( a\right) -r\left( a\right) <0$,
we may choose $a<b<1$ such that $H\left( a,b\right) <0$.
The latter inequality implies}
$W\left(-a,b\right) >W\left( -a,a\right) +W\left( a,b\right) \,.$
\qed
\end{nonsec}

\bigskip

Our next result refines, for $(2c-1)t>1\,,$ the upper bound in Proposition \ref{dhvProp}.
Because Lemma \ref{mfbd} will not be used and its proof is
straightforward and tedious, its proof is placed in an appendix at the
end of the paper.

\begin{lem} \label{mfbd}
  The inequality
  \begin{equation}\label{eq:newtarget}
    F_c(t)
   \leq \frac12\frac{t^2+(2c+1)t}{t+1},
  \end{equation}
  holds for  $ c>0 $ and $ t>0 $.
\end{lem}

The next result refines the lower bound of Proposition \ref{dhvProp}
for $c\ge 1, t\ge 0\,$ and the upper bound of
Lemma \ref{mfbd} for each $c>0$ and large enough $t\,.$

\begin{lem} \label{lem316}
  For $t\geq 0$, the following inequalities hold
  \begin{equation}\label{eq:lem311:1}
    \frac{t}{t+1}\log c+\frac{t}{2}\leq   F_c(t)\,,
    \text{ }c\geq 1\,,  \tag{1}
  \end{equation}%
  \begin{equation}\label{eq:lem311:2}
     F_c(t)\leq \log \Big(1+\frac{1}{4c}\Big)+\frac{t}{2},
    \text{ }c>0\,.  \tag{2}
  \end{equation}%
  Equality holds in (1) if and only if $t=0$, respectively in (2) if and only
  if $c\geq \frac{1}{2}$and $t=2\log (2c)\,.$ 
\end{lem}

\begin{proof}
(\ref{eq:lem311:1}) For each fixed $t>0$, we consider the expression
\begin{equation*}
  L_{t}(c)=1+c(e^{\frac{t}{2}}-e^{-\frac{t}{2}})-c^{\frac{t}{t+1}}e^{\frac{t}{2}}
\end{equation*}%
as a function of $c\,.$ 

The derivative is
\begin{equation*}
  L_{t}^{\prime }(c)=e^{\frac{t}{2}}-e^{-\frac{t}{2}}-\frac{t}{t+1}c^{-\frac{1%
  }{t+1}}e^{\frac{t}{2}},
\end{equation*}%
and the equation $L_{t}^{\prime }(c)=0$ has the unique solution
\begin{equation}
  c=c_{0}:=\Big(\frac{t}{t+1}\frac{1}{1-e^{-t}}\Big)^{t+1}.  \label{eq:c}
\end{equation}%
The inequality $e^{t}>t+1$ for $t>0\,,$ implies that $c_{0}<1$,
and hence $ L_{t}^{\prime }(c)>0$ holds for $c\geq 1$.
Since $L_{t}(1)=1-e^{-\frac{t}{2}}>0$, we have $L_{t}(c)>0$ for $c\geq 1$.

Hence the following inequality holds for $c\geq 1$ and $t>0$,
\begin{equation*}
  c^{\frac{t}{t+1}}e^{\frac{t}{2}}<1+c(e^{\frac{t}{2}}-e^{-\frac{t}{2}}).
\end{equation*}
Considering the logarithms of both sides, we have the assertion.

(\ref{eq:lem311:2}) By the arithmetic-geometric mean inequality,
$2\leq x+\frac{1}{x}$ holds for all $x>0$, hence
\begin{equation*}
  2-2ce^{-\frac{t}{2}}\leq \frac{1}{2c}e^{\frac{t}{2}}\,.
\end{equation*}
Adding $2ce^{\frac{t}{2}}$ to the both sides of this inequality, dividing by
$2$ and taking the logarithm, we obtain (\ref{eq:lem311:2}). 

Equality holds in (\ref{eq:lem311:2}) if and only if
$2ce^{-\frac{t}{2}}=1$, i.e. $t=2\log \left( 2c\right) $. 
\end{proof}


\begin{figure}[!htb]
    \centering
    \begin{minipage}{.5\textwidth}
        \centering
        \includegraphics[width=0.9\linewidth, height=0.2\textheight]{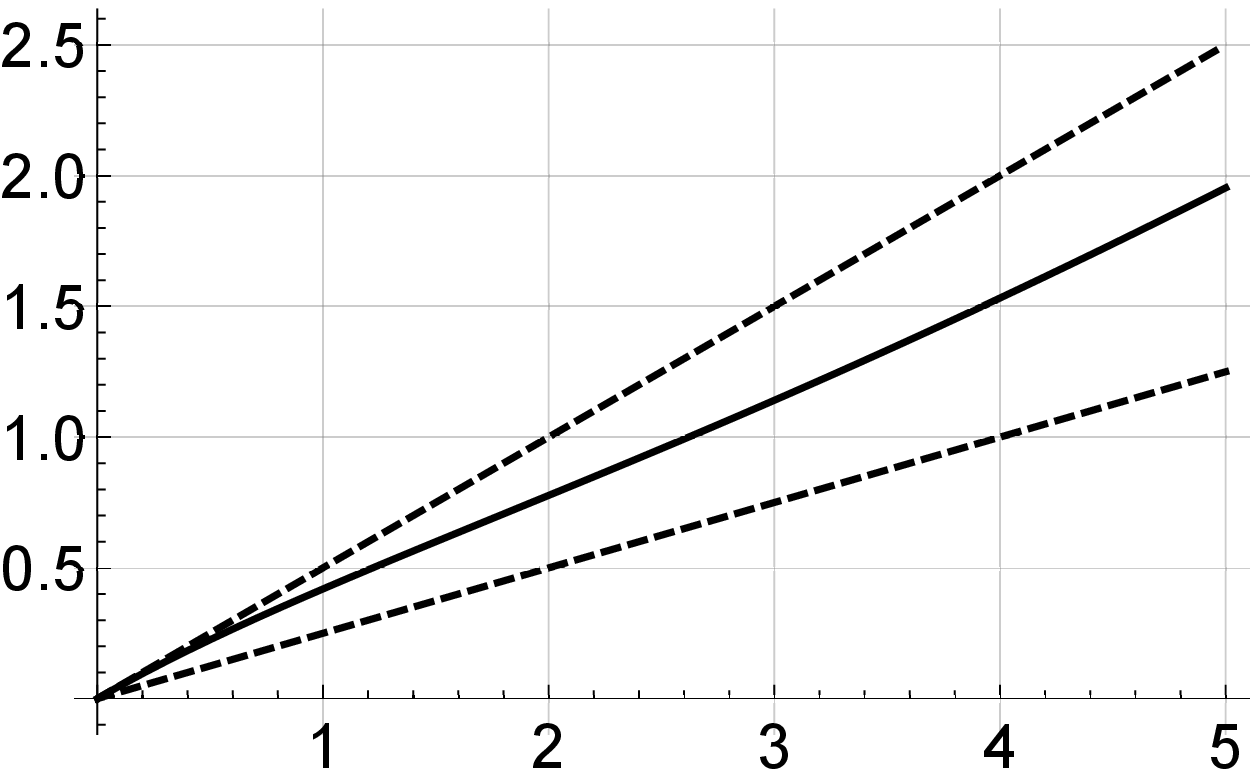}
    \end{minipage}%
    \begin{minipage}{0.5\textwidth}
        \centering
        \includegraphics[width=0.9\linewidth, height=0.2\textheight]{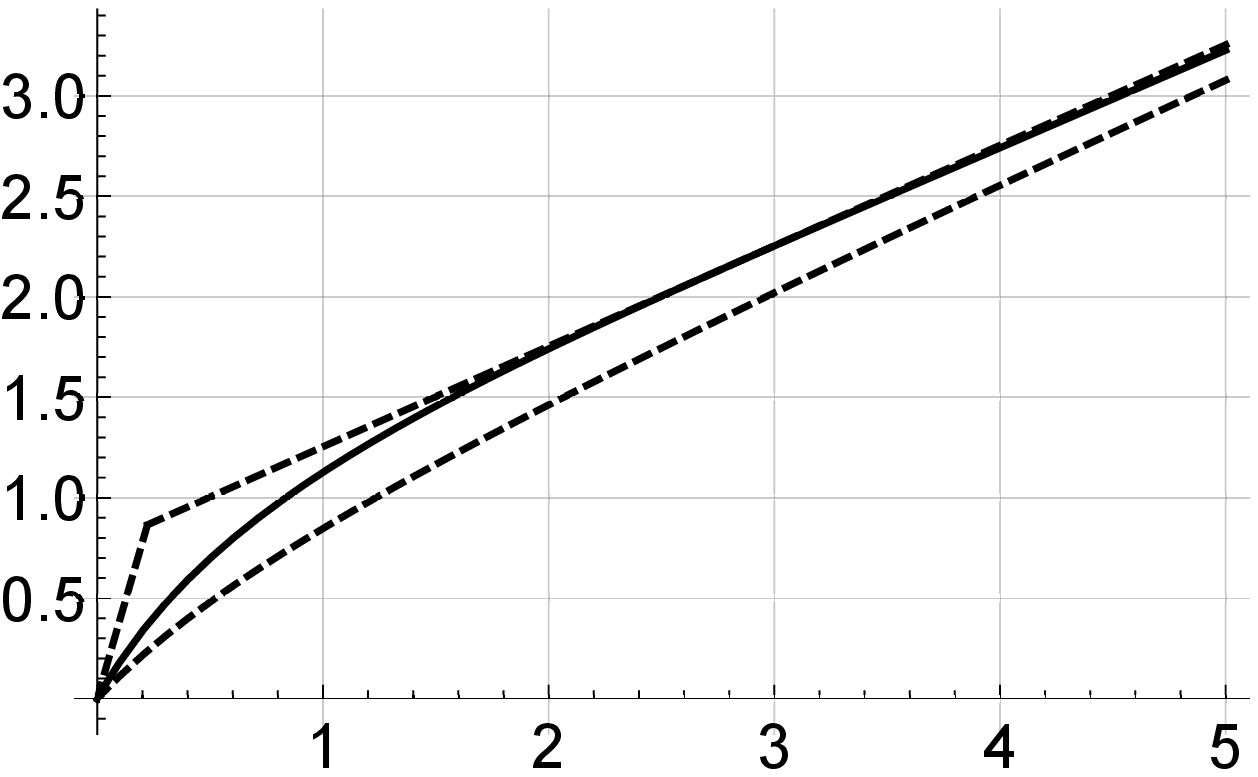}
    \end{minipage}
    \caption{The graphs of the functions $ y=l(c,t), y=F_c(t)$, and
           $ y=u(c,t)$ in Lemma \ref{Lemma313}.} Left: $c=0.5$, right: $c=2\,.$
  \label{fig:2}

\end{figure}

\begin{lem} \label{Lemma313}
  Let \renewcommand{\arraystretch}{2}
  \begin{equation}
    l(c,t)=\left\{
    \begin{array}{ll}
     \dfrac{t}{t+1}\log c+\dfrac{t}{2}\quad  & \mbox{if \ } c\geq 1,\ t\geq 0 \\
     \dfrac{ct}{2} & \mbox{if \ } 0<c\leq 1,\ t\geq 0%
    \end{array}%
    \right.   \label{eq:lower}
  \end{equation}%
  \renewcommand{\arraystretch}{1} and
  \begin{equation}
    u(c,t)=\min \Big\{\log \Big(c+\frac{1}{4c}\Big)+\frac{t}{2},\ \ \log
    (1+ct)+c(e^{t}-1)\Big\}.  \label{eq:upper}
  \end{equation}

  Then, for $t>0$ and $c>0$, the following inequalities hold
  \begin{equation}
    l(c,t)<  F_c(t)  \leq u(c,t).  \label{eq:bd}
  \end{equation}

  The upper bound is attained, for $t>0$ and $c>0,$ if and only if
  $c>\frac{1}{2}$ and $t=2\log (2c).$
\end{lem}

\begin{proof}
We first prove the lower bounds. By Lemma \ref{lem316} it is enough to prove
the case $0<c\leq 1.$ For each fixed $t>0$, let
\begin{equation*}
  f_{t}(c)=1+c(e^{\frac{t}{2}}-e^{-\frac{t}{2}})-e^{\frac{ct}{2}}.
\end{equation*}%
Then $f_{t}(0)=0$ and
\begin{equation*}
  f_{t}^{\prime }(c)=e^{\frac{t}{2}}-e^{-\frac{t}{2}}-\frac{t}{2}e^{\frac{ct}{2}}
\end{equation*}%
is decreasing on the real axis.
Because $f_{t}^{\prime }(0)=e^{\frac{t}{2}}-e^{-\frac{t}{2}}-\frac{t}{2}>0$,
the sign of $f_{t}^{\prime }(1)=e^{\frac{t}{2}}\left( 1-\frac{t}{2}-e^{-t}\right) $
depends on $t$. The equation $1-\frac{t}{2}-e^{-t}=0$
has a unique positive solution $t_{0}\in \left(1,2\right) $
and we have  $f_{t}^{\prime }(1)>0$ for $0<t<t_{0}$ and $f_{t}^{\prime }(1)<0$
for $t>t_{0}$.

If $0<t\leq t_{0}$, it follows that $f_{t}$ is increasing on $\left[ 0,1%
\right] $, therefore $f_{t}(c)>0$ for $0<c\leq 1$.

If $t>t_{0}$,  there is a unique positive critical point $c=c_{0}$ as a
solution of
\begin{equation*}
  f_{t}^{\prime }(c)=0\,.
\end{equation*}
The function $f_{t}$ attains the maximum at $c=c_{0}$. Since $f_{t}(0)=0$
and $f_{t}(1)=1-e^{-\frac{t}{2}}>0$, we have $f_{t}(c)>0$ for $0<c\leq 1$.
Therefore
\begin{equation*}
  e^{\frac{ct}{2}}<1+2c\sinh \frac{t}{2}
\end{equation*}
holds for $0<c\leq 1$ which yields the desired inequality.

\bigskip

We now prove the upper bound. By Lemma \ref{lem316} it is enough to prove
that 
\begin{equation*}
  \log \Big(1+2c\sinh \frac{t}{2}\Big)<\log (1+ct)+c(e^{t}-1)
\end{equation*}
holds for $c>0$ and $t>0.$ For each fixed $t>0\,,$ let
\begin{equation*}
  g_{t}(c)=(1+ct)e^{c(e^{t}-1)}-1-c(e^{\frac{t}{2}}-e^{-\frac{t}{2}})\,.
\end{equation*}
Then, $g_{t}(c)$ is increasing on $(0,\infty )$, because
\begin{equation*}
  g_{t}^{\prime }(c)=e^{c(e^{t}-1)}\big((1+ct)(e^{t}-1)+t\big)-(e^{\frac{t}{2}%
  }-e^{-\frac{t}{2}})
\end{equation*}
is clearly increasing and
\begin{equation*}
  g_{t}^{\prime }(0)=e^{t}-1+t-(e^{\frac{t}{2}}-e^{-\frac{t}{2}})=(e^{\frac{t}{%
  2}}-e^{-\frac{t}{2}})(e^{\frac{t}{2}}-1)+t>0\,.
\end{equation*}
Therefore, $g_{t}(c)>0$ holds as $g_{t}(0)=0$. Hence, we have
\begin{equation*}
  1+c(e^{\frac{t}{2}}-e^{-\frac{t}{2}})<(1+ct)e^{c(e^{t}-1)}\,,
\end{equation*}
which yields the assertion.

Moreover, for $t>0$, $c>0$, the equation $\log \Big(1+2c\sinh \dfrac{t}{2}\Big)=u(c,t)$
leads  to
$$
  \log \Big(1+2c\sinh \dfrac{t}{2}\Big)
  =\log \Big(c+\frac{1}{4c}\Big)+\frac{t}{2}<\log (1+ct)+c(e^{t}-1),
$$
  which holds if and only if $c>\frac{1}{2}$ and $t=2\log (2c)$.
\end{proof}

\medskip

\begin{nonsec}{\bf Proof of Theorem \ref{thm110}.} {\rm
The upper bound follows from Lemmas \ref{dhv44}, \ref{lem316}
and the lower bound from Lemma \ref{lem316}.} \qed 
\end{nonsec}

\medskip

The above results readily give the following theorem.

\begin{thm}
  For points $x,y \in G,$ and a number $ c \ge 1\,,$ we have
  $$
    L \, j_G(x,y) \le W(x,y) \le U\, j_G(x,y)
  $$
  where $W$ is the metric
  $$
    W(x,y)= \log \Big( 1+ 2c \sinh \frac{j_G(x,y)}{2}\Big)
  $$
  and
  $$
    L= \frac{1}{2} +\frac{\log(c)}{1+ j_G(x,y)}\,, \quad U
     = \frac{j_G(x,y)+(2c+1)}{2(1+ j_G(x,y))} \,.
  $$
\end{thm}

\begin{prop}
  Let $\left( G,\rho _{G}\right) $ and $(D,\rho _{D})$ be two metric spaces
  and let  $\omega _{G,c}=\log \left( 1+2c\sinh \frac{\rho _{G}}{2}\right) $
  and $\omega _{D,c}=\log \left( 1+2c\sinh \frac{\rho_{D}}{2}\right) $,
  where $c\geq 1$. If $f\colon (G,\rho _{G})\rightarrow (D,\rho _{D})$ is an 
  $L$-Lipschitz function, then 
  $f\colon (G,\omega _{G,c})\rightarrow (D,\omega _{D,c})$
  is $L^{\prime }$-Lipschitz, where $L^{\prime }=L$ if  $L\geq 1$ and 
  $L^{\prime }=cL$ if $L>0$. Conversely, 
  if  $f\colon (G,\omega _{G,c})\rightarrow (D,\omega _{D,c})$ is an 
  $L^{\prime }$-Lipschitz function with $L^{\prime }>0 $, 
  then $f\colon (G,\rho _{G})\rightarrow (D,\rho _{D})$ is an 
  $2cL^{\prime }$-Lipschitz function. 
\end{prop}

\begin{proof}
Denote $F_{c}(t)=\log \big(1+2c\sinh \left( \frac{t}{2}\right) \big)$
as in Lemma \ref{propmm}. 
By Theorem \ref{dhvfun}, $\omega _{G,c}=F_{c}\circ \rho _{G}$
and $\omega _{D,c}=F_{c}\circ \rho _{D}$ are metrics on $G$ and $D$,
respectively.

Fix distinct points $x,y\in G\,.$

Assume that $f\colon (G,\rho _{G})\rightarrow (D,\rho _{D})$\textit{\ }is 
$L$-Lipschitz. We have to prove that 
\begin{equation}
  F_{c}\big(\rho _{D}(f(x),f(y))\big)
     \leq L^{\prime }F_{c}\big( \rho _{G}(x,y)\big),  
  \label{Lipfin}
\end{equation}%
where $L^{\prime }=L$ if $L\geq 1$ and $L^{\prime }=cL$ whenever $L>0$. 

Since $f\colon (G,\rho _{G})\rightarrow (D,\rho _{D})$\textit{\ }is $L$-Lipschitz
and $F_{c}$ is increasing, 
\begin{equation}
  F_{c}\big(\rho _{D}(f(x),f(y))\big)
   \leq F_{c}\big( L\rho _{G}(x,y)\big) \,.
\label{Lipinit}
\end{equation}

Assume first that $L\geq 1$. By Lemma \ref{propmm}, $\frac{F_{c}(t)}{t}$ is
decreasing on $\left( 0,\infty \right) $, therefore $F_{c}(Lt)\leq LF_{c}(t)$
for all $t>0$, as $L\geq 1$. 
Then $F_{c}\big( L\rho _{G}(x,y)\big) \leq LF_{c}\big( \rho _{G}(x,y)\big) $.
The latter inequality and (\ref{Lipinit}) imply 
(\ref{Lipfin}) with $L^{\prime }=L$. 

By Lemma \ref{propmm}, $\frac{t}{2}\leq F_{c}(t)\leq ct$ for all $t\geq 0$,
if $c\geq 1$. 

For all $L>0$, (\ref{Lipinit}) implies     
\begin{equation*}
  F_{c}\big(\rho _{D}(f(x),f(y))\big)\leq F_{c}\big( L\rho _{G}(x,y)\big) \leq
  cL\rho _{G}\left( x,y\right) .
\end{equation*}

Now assume that (\ref{Lipfin}) holds. Then 
\begin{align*}
  \rho _{D}(f(x),f(y)) 
    &\leq 2F_{c}\big(\rho _{D}(f(x),f(y))\big)
            \leq L^{\prime}F_{c}\big( \rho _{G}(x,y)\big)  \\
    &\leq  2cL^{\prime }\rho _{G}(x,y). 
\end{align*}
\end{proof}


\section{Metrics and quasiregular maps}


If $D\in \{\mathbb{B}^{n},\mathbb{H}^{n}\}$ and $\rho _{D}$ is the
hyperbolic metric on $D$, then the metric defined on $D$ by 
$W_{c}(x,y)=\log\left( 1+2c\sinh \frac{\rho _{D}\left( x,y\right) }{2}\right) $,
where $c\geq 1$, is invariant under M\"{o}bius self maps of $D$, 
due to the M\"{o}bius invariance of the hyperbolic metric. 

\medskip

We also recall some notation about special functions and
the fundamental distortion result of
quasiregular maps, a variant of the Schwarz lemma for these maps, see
\cite{hkv}.
For $r\in(0,1)$ and $K>0$, we define the distortion function
$$\varphi_K(r)=\mu^{-1}(\mu(r)/K),$$
where $\mu(r)$ is the modulus of the planar Gr\"otzsch ring, a decreasing
homeomorphism $\mu\colon (0,1)\to (0,\infty)\,,$
see \cite[pp. 92-94]{avv}, \cite[pp.120-125]{hkv}. 

\begin{thm}\label{NewThm}
  Let $G_1$ and $G_2$ be simply-connected domains in 
  $\mathbb{R}^2$ and let $f\colon G_1\to G_2=f(G_1)$ be a 
  $K$-quasiregular mapping.
  Then  for all $x,y\in G_1$ 
  \begin{align*}
    \rho_{G_2}(f(x),f(y))\leq c(K)\max\{\rho_{G_1}(x,y),
    \rho_{G_1}(x,y)^{1\slash K}\} 
  \end{align*}
 where $c(K)$ is as in
  \emph{\cite[Thm 16.39, p. 313]{hkv}, \cite[Theorem 3.6]{wv}}.
\end{thm}

\begin{rem}
  {\rm By  \cite[Thm 16.39, p. 313]{hkv},
 \begin{align*}
  K\leq c(K)\leq\log(2(1+\sqrt{1-1\slash e^2}))(K-1)+K
 \end{align*}
 and, in particular, $c(K)\to1$, when $K\to1$.}
\end{rem}

\begin{nonsec}{\bf Proof of Theorem \ref{Wqc}.}  {\rm
By Theorem \ref{NewThm} \cite[Thm 16.39]{hkv},
\begin{equation}
  \rho \left( f\left( x\right) ,f\left( y\right) \right) \leq c(K)\max \left\{
  \rho \left( x,y\right) ^{1/K},\rho \left( x,y\right) \right\} .
  \label{dishyp}
\end{equation}
According to Lemma \ref{propmm},

\begin{equation}
  \frac{t}{2}
  <\log \left( 1+2\lambda \sinh \frac{t}{2}\right)
  <\lambda t \text{\,\, for every } \,\, t\in \left( 0,\infty \right) \text{.}
  \label{Lemma3.4}
\end{equation}
Then for all $x,y\in \mathbb{B}^{2}$
\begin{align*}
  W_{\lambda }\left( f\left( x\right) ,f\left( y\right) \right)
    &\leq \lambda \rho (f(x),f(y))
     \leq \lambda c(K)\max \left\{ \rho \left(x,y\right) ^{1/K},
      \rho \left( x,y\right) \right\} \\
   &\leq \lambda c(K)\max \left\{ 2^{1/K}W_{\lambda }\left( x,y\right)^{1/K},
      2W_{\lambda }\left( x,y\right) \right\}
\end{align*}
and (\ref{distow}) follows.
\qed
}
\end{nonsec}

\section{Appendix}

We give here the proof of Lemma \ref{mfbd}.
We shall apply the inequality

\begin{equation}\label{eq:et}
   1-e^{-t}=\frac{e^t-1}{e^t}> \frac{t}{t+1}\,
\end{equation}
which easily  follows from $ e^{t}-(t+1)>0 \,, t> 0\,.$

\bigskip

\begin{nonsec}{\bf Proof of Lemma \ref{mfbd}.} {\rm
The right hand side of \eqref{eq:newtarget} can be written as
$ \dfrac12 t+\dfrac{ct}{t+1} $.
Taking the exponential function of both sides of \eqref{eq:newtarget},
we need to check that for each fixed  $ t >0 $ the following inequality
holds,
\begin{equation}\label{eq:expform}
  E_t(c):=e^{\frac{t}2}\cdot e^{\frac{ct}{t+1}}
               -1-c(e^{\frac{t}2}-e^{-\frac{t}2})>0, \quad (c>0).
\end{equation}

\medskip

First, we remark that $ E_t(0) =e^{\frac{t}2}-1> 0 $ holds for $ t>0 $.

Next, we will show that $ E'_t(c) >0 $ holds for $ c>0 $.
It is clear that the derivative
\begin{equation}\label{eq:diffE}
  E'_t(c)=e^{\frac{t}{2}}\Big(\frac{t}{t+1}e^{\frac{t}{t+1}c}-1+e^{-t}\Big)
\end{equation}
is an increasing function with respect to $ c $, and satisfies
$ \lim_{c\to \infty}E'_t(c)=\infty $.
From \eqref{eq:et}, we have
$$
   E'_t(0)=e^{\frac{t}{2}}\Big(\frac{t}{t+1}-(1-e^{-t})\Big)<0.
$$
Hence, $ E'_t(c)=0 $ has the unique positive root $ c_0 $,
and $ E_t(c) $ attains the minimum at $ c=c_0 $.
Moreover, the root $ c_0 $ is given by the formula
$$
    c_0=\frac{t+1}{t}\log\Big(\frac{t+1}t(1-e^{-t})\Big).
$$
Then,
$$
  E_t(c_0)=\frac{t+1}t\Big(1-\log\big(\frac{t+1}t(1-e^{-t})\big)\Big)
         (e^{\frac{t}2}-e^{-\frac{t}2})-1.
$$
Therefore, we need to show that
\begin{equation}\label{eq:min}
  C(t):=\Big(1-\log\big(\frac{t+1}t(1-e^{-t})\big)\Big)
         (e^{\frac{t}2}-e^{-\frac{t}2})-\frac{t}{t+1}>0 \quad (t>0).
\end{equation}
By using \eqref{eq:et}, we have
\begin{align*}
  C(t) &\geq \Big(1-\log\big(\frac{t+1}t(1-e^{-t})\big)\Big)
           (e^{\frac{t}2}-e^{-\frac{t}2})-\frac{e^{t}-1}{e^t}\\
       &= \Big(1-\log\big(\frac{t+1}t(1-e^{-t})\big)-e^{-\frac{t}2}\Big)
           (e^{\frac{t}2}-e^{-\frac{t}2}).
\end{align*}
Since $e^{\frac{t}2}-e^{-\frac{t}2} > 0 $, we have to check that
\begin{align*}
   \widetilde{C}(t)&:=1-\log\Big(\frac{t+1}t(1-e^{-t})\Big)-e^{-\frac{t}2}\\
                   &=1-\log (t+1)+\log t-\log (e^{t}-1)+t-e^{-\frac{t}{2}}
                   >0\,.
\end{align*}
Thus $ \widetilde{C}(0)=1-\log 1-e^0=0 $
(note that $ \displaystyle \lim_{t\to 0}\dfrac{1-e^{-t}}{t}=1 $),
and
\begin{align*}
    \widetilde{C}'\left( t\right)
     & =\frac{1}{t}-\frac{1}{t+1}-\frac{e^{t}}{e^{t}-1}+1
       +\frac{1}{2}e^{-\frac{t}{2}}\\
     &=\frac{2\left( e^{t}-1\right)
       -2t\left( t+1\right) +t(t+1)\left( e^{\frac{t}{2}}
       -e^{-\frac{t}{2}}\right) }{2t\left( t+1\right) \left( e^{t}-1\right)}.
\end{align*}
Applying $ 2t(t+1)(e^t-1)>0 $, we need to check that
$$
 \widehat{C}(t)
     :=2\left( e^{t}-1\right) -2t\left( t+1\right)
       +t(t+1)\left( e^{\frac{t}{2}}-e^{-\frac{t}{2}}\right)>0.
$$
Then, $\widehat{C}(0)=0 $ and
$$
    \widehat{C}'(t)=2e^{t}-2\left( 2t+1\right)
       +\left(2t+1\right) \left( e^{\frac{t}{2}}-e^{-\frac{t}{2}}\right)
       +\frac{t\left(t+1\right) }{2}\left( e^{\frac{t}{2}}
       +e^{-\frac{t}{2}}\right).
$$
Similarly, $ \widehat{C}'(0)=0 $ and
\begin{align*}
    \widehat{C}''(t)
    = &2\left( e^{t}-1\right)
      +\left( e^{\frac{t}{2}}+e^{-\frac{t}{2}}-2\right)
      +2t\left( e^{\frac{t}{2}}+e^{-\frac{t}{2}}\right)\\
      &+\frac{t\left( t+1\right)+8}{4}
       \left( e^{\frac{t}{2}}-e^{-\frac{t}{2}}\right).
\end{align*}
Therefore, $ \widehat{C}''(t)>0 $,
since the coefficient of each term is positive
for $ t>0 $.
Then, $ \widehat{C}'(t)>0 $, since
$ \widehat{C}'(t) $ is increasing function with $ \widehat{C}'(0)=0 $.
Similarly, $ \widehat{C}(t)>0 $.

Hence
$\widetilde{C}\left( t\right)
 >\underset{s\searrow 0}{\lim }\widetilde{C}\left( s\right) =0$
for all $t>0$.

In conclusion,  we have $ C(t)>0 $ for each $ t>0 $.
Finally, we have $ E_t(c)>0 \ (c>0) $ for each $ t>0 $,
and we have the assertion.}
\qed
\end{nonsec}

\textbf{Acknowledgments.}

 This work was partially supported by JSPS KAKENHI
    Grant Number 19K03531 and by JSPS Grant BR171101.


\nocite{wv}

\bibliographystyle{spmpsci}      


\end{document}